\documentclass[11pt]{amsart}
\usepackage{amsfonts}

\usepackage[all,cmtip]{xy}
\usepackage{amsmath}
\usepackage{amsthm}
\usepackage{amssymb}
\usepackage{enumitem}
\newtheorem{thm}{Theorem}[section]
\newtheorem{conj}{Conjecture}[section]
\newtheorem{prop}[thm]{Proposition}
\newtheorem*{definition*}         {Definition}
\newtheorem{lemma}[thm]{Lemma}

\theoremstyle{remark}

\newcommand*{\Q}{\mathbb{Q}}
\newcommand*{\Hh}{\mathbb{H}}

\newcommand*{\Z}{\mathbb{Z}}

\newcommand*{\R}{\mathbb{R}}

\newcommand*{\C}{\mathbb{C}}

\newcommand*{\PP}{\mathbb{P}}
\newcommand*{\FF}{\mathcal{F}}

\newcommand*{\Vol}{\textrm{Vol}}

\newcommand*{\Mob}{\textrm{Mob}}
\newcommand*{\ra}{\rightarrow}
\def\SL{{\rm SL}}

\def\sl{{\rm sl}}
\def\n{{\rm n}}

\usepackage{amscd,amssymb,amsmath}
\usepackage{color}

\author{Jonathan Pila and Jacob Tsimerman}
\address{Mathematical Institute, University of Oxford, UK}
\address{Department of Mathematics, University of Toronto, Canada}

\begin{document}

\begin{abstract}

In this paper we prove a functional transcendence statement for the
$j$-function which is
an analogue of the Ax-Schanuel theorem
for the exponential function. It asserts, roughly, that
atypical algebraic relations among functions and
their compositions with the
$j$-function are governed by modular relations.

\end{abstract}

\title{Ax-Schanuel for the $j$-function}
\maketitle

\centerline{\it To Peter Sarnak on the occasion of his 61st birthday}

\section{Setup and main theorem}

Let $j:\Hh\rightarrow \C$ denote the classical modular function,
or {\it $j$-function\/}, by means of which the quotient $Y(1)$
of the complex upper-half plane $\Hh$ by the modular group
is identified with $\C$. Here $Y(1)$ is the moduli space of elliptic curves
over $\C$ up  to isomorphism. Contrary to some conventions, we shall consider $\Hh$ as an open
subset of $\PP^1(\C)$ rather than of $\C$. This is more natural for us since $\SL_2(\C)$ acts on $\PP^1(\C)$. From now on,
we write $\PP^1$ to denote $\PP^1(\C)$ for brevity.
By abuse of notation, we shall refer to the cartesian powers of this map
as $j$ regardless of how many variables are involved.
For a connected open domain $D$ of a complex algebraic variety $V$, We will refer to a complex analytically
irreducible component  of the intersection of an algebraic variety of $V$
with $B$ as an {\it algebraic subvariety\/} of $B$. Let $\Gamma$ denote the
graph of the $j$ function in $\Hh^n\times Y(1)^n$. The definition of a
{\it weakly special subvariety\/}
is given in \S2. A first version of our result is then the following.

\begin{thm}[Ax-Schanuel for $j(z)$]\label{main0}

Let $V\subset (\PP^1)^n\times Y(1)^n$ be an algebraic subvariety,
and let $U$ be an irreducible component of $V\cap \Gamma$.
If the projection of $U$ to $Y(1)^n$ is not contained in a proper
weakly special subvariety, then $\dim V = \dim U  +n$.
\end{thm}

If we replace $j:\Hh^n\rightarrow \C^n$ by $\exp: \C^n\rightarrow (\C^\times)^n$,
the cartesian power of the exponential function, and take weakly special
subvarieties in $(\C^\times)^n$ to be cosets of algebraic subtori, then this
is a formulation of ``Ax-Schanuel'' for the exponential function,
as may be found in Ax \cite{Ax2}. The more usual statement of
Ax's theorem \cite{Ax1} is in the setting of a differential field.
We will give a statement in that setting below, though in a stronger form.
The connection between the two settings is afforded by the Seidenberg
embedding theorem, by means of which differential fields of
characteristic zero
may be embedded in fields of germs of meromorphic functions. We indicate
in \S2 how to deduce the differential version of our theorem from the
complex version.

\bigbreak

Theorem 1.1 generalises the ``Ax-Lindemann'' theorem for $j(z)$
established  in \cite{P1} in connection with the Andr\'e-Oort conjecture for
$Y(1)^n$.
We expect that Theorem 1.1 will be useful in connection with the broader
Zilber-Pink conjecture for $Y(1)^n$: a preliminary application is
given in section 7.
On the Zilber-Pink conjecture generally see \cite{BMZ, Pn1, Pn2, Z, Zi}.
A special case of Theorem 1.1 (``Modular Ax-Logarithms'') is used by
Habegger-Pila \cite{HP1} in affirming a very special case of the
Zilber-Pink conjecture
for $Y(1)^n$. A more general result (but less general than 1.1) appears as
one of two hypotheses in a conditional affirmation of Zilber-Pink for $Y(1)^n$
in \cite{HP2}.  Related functional transcendence results in connection with
Zilber-Pink have also appeared in the work of
Betrand-Masser-Pillay-Zannier \cite{MZ, BMPZ} on ``torsion anomalous'' points.

However, we want to frame a stronger result that is
a transcendence theorem not just for $j(z)$
but also for $j'(z)$ and $j''(z)$. Including further derivatives gains
nothing as $j'''(z)$ is algebraic over $j(z),j'(z)$ and $j''(z)$ (see \S2.4).
To state a theorem encoding the derivatives geometrically we need
the concept of a \emph{jet space}.

\begin{definition*} For a complex analytic manifold $M$ and a positive integer
$k$, the $k$'th jet space $J_kM$ is a fiber bundle over $M$ whose fiber
over $m\in M$ consists of equivalence classes of germs of
holomorphic maps from a neighborhood of $0\in\C$ to $M$ taking
$0$ to $m$, where we declare 2 such maps equivalent if they agree up
to order $k$. For example, $J_1M$ is the tangent bundle of $M$.
\end{definition*}

A map $M\rightarrow N$ induces naturally a map $J_kM\rightarrow J_kN$.
We let $J^0_kM\subset J_kM$ denote the sub-space consisting of germs of
functions with vanishing first derivative. If $M$ is an algebraic variety, then
$J_kM$ can be naturally given such a structure as well. For a complex
open disc $D$ parameterized by a holomorphic variable $z$, we can
give local co-ordinates for $J_k D\cong D\times \C^k$ by letting
$(z,a_1,a_2,..,a_k)$ denote the germ at $0$ of the function
$$f(w)=z+\sum_{i=1}^k a_kw^k/k!.$$ In these co-ordinates, $J^0_kM$
is locally defined by the equation $a_1=0$.
An {\it algebraic subvariety\/} of  $J_k\Hh^n\times J_kY(1)^n$ will again mean
a complex analytically irreducible component of the intersection with
$J_k\Hh^n\times J_kY(1)^n$ with an algebraic subvariety in the ambient complex
affine space.
For $k\ge 0, n\ge 1$ let $J_kj:J_k\Hh^n\rightarrow J_kY(1)^n$ denote
the map induced by $j$,
and let $\Gamma^n_k$ denote the graph of this map. When $k,n$ are fixed we
will just refer to the graph as $\Gamma$. We

\begin{thm}[Ax-Schanuel with derivatives for $j(z)$]\label{main}

Let $V\subset J_2\Hh^n\times J_2Y(1)^n$ be an algebraic subvariety,
and let $U$ be an irreducible component of $V\cap \Gamma$.
If the projection of $U$ to $Y(1)^n$ is not contained in a weakly special
subvariety, and none of the projections of $U$ to $J_2\Hh$ are contained
inside $J^0_2\Hh$, then $\dim V = \dim U  + 3n$.
\end{thm}

In fact, the above theorem is equivalent to the following purely
differential-algebraic version.
Let $K$ be a characteristic 0 differential field with $m$ commuting derivations
$D_i$. Let $C=\bigcap_i\textrm{ker}D_i$ be the constant field of $K$.
In the following theorem,
$F(j, j', j'', j''')=0$
is the algebraic  differential equation
satisfied by $j(z)$ (see \S2.4),
$\Phi_N$ are the {\it modular polynomials\/},
and  the {\it rank\/} of a matrix is over $K$.

\begin{thm}\label{main2}
Let $z_i,j_i,j'_i,j''_i, j_i'''\in K^\times$, $i=1,\ldots, n$,
such that
$$D_kj_i=j_i'D_kz_i,\quad D_kj'_i=j_i''D_kz_i, {\rm\ and\ } D_kj_i''=j_i'''D_kz_i$$
for all $i$ and $k$.
Suppose further that $F(j_i,j'_i,j''_i,j_i''')=0$ for $i=1,\ldots, n$,
and that $\Phi_N(j_i,j_k)\neq 0$ for all $i, k, N$ and $j_i\notin C$ for all $i$.
Then
$${\rm tr. deg.}_{C}C(z_1,j_1,j'_1,j''_1,\ldots, z_n,j_n, j_n',j''_n)\geq 3n +
{\rm rank}\left(D_k z_i\right)_{i,k}.$$
\end{thm}

We show in \S2 that Theorem 1.2 implies Theorem 1.3 (and {\it vice versa\/}).
The rest of the paper, except section 7, is devoted to the proof of
Theorem \ref{main} and we adopt its hypotheses throughout. Thus
$V\subset J_2\Hh^n\times J_2Y(1)^n$ is an algebraic variety,
and $U$ is an irreducible component of $V\cap \Gamma$.
We assume that the projection of $U$ to $Y(1)^n$ is not contained in a proper
weakly special subvariety, and none of the projections of $U$ to $J_2\Hh$ are
contained inside $J^0_2\Hh$.

Amongst other ingredients, we rely on the theory of o-minimality.
We will say that a set is {\it definable\/} if it is definable in
the o-minimal structure $\R_{\rm an,\ exp}$. For an introduction to this theory, see
\cite{D}; on  $\R_{\rm an,\ exp}$ see \cite{DMM, DM}.
The second author \cite{T} gave a proof of the exponential Ax-Schanuel theorem
using o-minimality.

After some preliminary results are gathered in \S2,
and some initial ``conditioning'' of $V$ is undertaken in \S3,
the main elements of the proof can be summarized as follows.
We study a certain definable
set $Z$ of elements $\gamma\in{\rm SL}_2(\R)^n$ such that $\gamma V$
and a suitable component $U_\gamma$ of $\gamma V\cap \Gamma$
have the same properties as $V$ and $U$.
Volume estimates
of Hwang-To \cite{HT} enable us in \S4 to show that $U$ is present over ``many''
fundamental domains $\gamma \mathcal{F}^n$
for the action of ${\rm SL}_2(\Z)^n$ on
$\Hh^n$. Here $\mathcal{F}$ is the usual fundamental domain
for the action of ${\rm SL}_2(\Z)$ on $\Hh$.
Thus $Z$ has ``many'' integer points.
An application of the Counting Theorem of Pila-Wilkie \cite{PW} now shows
that $Z$ contains algebraic curves.
With these we may seek to increase the dimension of $V$ and $U$, or
to reduce the difference $\dim V-\dim U$, which we handle by induction.
Eventually, we reduce to the situation where $V$ has a ``large'' stabilizer.
Studying the  possible structure of the Zariski closure of the integer
points of the stabilizer of $V$ allows us to apply an o-minimal version of
Chow's theorem due to Peterzil-Starchenko \cite{PS1, PS2} to prove
that the image of $U$ in $Y(1)^n$ is algebraic. It is in this step only
that our restriction to powers of the modular curve,
rather than a general Shimura variety, is essential, as $\SL_2(\R)^n$ has
relatively few subgroups. Finally, we use monodromy results
due to Deligne-Andr\'e \cite{An, De} to complete the proof.

Suitable ``Ax-Schanuel''  formulations should  hold for any (mixed) Shimura
variety. In particular for $\mathcal{A}_g$, generalising ``Ax-Lindemann''
\cite{PT}, we expect the following.
\begin{conj}
Let $\pi: \Hh_g\rightarrow \mathcal{A}_g$ be the classical uniformisation
of the moduli space of principally polarised abelian varieties of dimension $g$.
Let $\Gamma$ be the cartesian power of the graph of $\pi$ in
$\Hh_g^n\times \mathcal{A}_g^n$.
Let $V\subset \Hh_g^n\times \mathcal{A}_g^n$
be an algebraic variety and $U$ an irreducible component of $V\cap\Gamma$.
If the projection of $U$
to $\mathcal{A}_g$ is not contained in a proper weakly special subvariety then
$\dim V =\dim U +n\frac{g(g+1)}{2}$.
\end{conj}

In fact, we can be more ambitious and formulate a version with derivatives
for a general Shimura variety. Consider a Shimura variety $S$
corresponding to a \emph{semisimple} real group $G$ and the cover
$\pi:\Hh\ra S$. Look at the map
$J_k\pi:J_k\Hh\ra J_kS$ for $k\geq 2$. Consider the graph $\Gamma^S_k$,
and its Zariski closure $Z^S_k$.
For $S=\mathcal{A}_g$ we have $\dim Z^S_k-\dim\Gamma^S_k = \dim G$
by results of Bertrand-Zudilin \cite{BZ}. This suggests the following formulation.

\begin{conj}
Let $\pi:\Hh\ra S$ be the uniformizing maps of a Shimura variety $S$,
corresponding to a semisimple group $G$. Let $k\geq 2$ and
$V\subset J_k\Hh\times J_k S$ be a Shimura variety whose projection
to $S$ is not contained
in a weakly special subvariety. Moreover, assume that for a generic point
$v\in V$ the projection of $v$ to $J_1 S$ does not lie in the image of
any weakly special subvariety.
Let $U$ be a positive dimensional component of $V\cap \Gamma^S_k$.
Then $\dim V = \dim U + \dim G$.
\end{conj}

Let us explain the appearance of the assumption ``the projection to $J_1 S$ does not lie in a weakly special subvariety''. Suppose $\Hh$ splits as $\Hh_1\times \Hh_2$, and consider a point $x\in J_1 \Hh$. This corresponds to a point $(h_1,h_2)$ with a tangent vector $(v_1,v_2)$. Now $G(\C)$ acts on $J_1\Hh$, and if, say, $v_1=0$ then the action of $G(\C)$ preserves this property. So, the variety $V$ given by insisting that the projections to $J_1\Hh_1$ as well as its image in the quotient have vanishing tangent vector should yield a counterexample.
The assumption in the conjecture is designed to prevent this from happening.

\section{Preliminaries}

\subsection{Notation}

We let $\mathcal{F}$ denote the usual fundamental domain for the
action of the modular group ${\rm SL}_2(\Z)$ on $\Hh$.

The inclusion $\mathcal{F}\subset \Hh$ gives and inclusion of jet spaces and we
identify $J_k\mathcal{F}$ with its image in $J_k\mathcal\Hh$. We denote
$J_2\mathcal{F}\times J_2Y(1)$ by $\overline{\mathcal{F}}$.

The group ${\rm SL}_2(\R)$ acts on $\Hh$ by fractional linear transformations,
inducing an action on $J_2\Hh$. We extend this action to
$J_2\Hh\times J_2 Y(1)^n$ by setting the action to be trivial on $J_2Y(1)^n$.
We will
have use for the extension of this action to $\SL_2(\C)$ acting on $\PP^1$
and $J_2\PP^1$.

Then $\overline{\mathcal{F}}$ is a fundamental domain for the action
of ${\rm SL}_2(\Z)$ on $J_2\mathcal\Hh\times J_2Y(1)$.
The reader may also check that the action of ${\rm SL}_2(\Z)$ preserves
the graph of $J_2j$.

\subsection{Definability}
We use the basic observation, made by Peterzil and Starchenko, that
the $j$-function restricted to $\mathcal{F}$ is definable. Then the graph
of $J_2j$ in $J_2\mathcal{F}\times J_2Y(1)$ is likewise definable, along with
its cartesian products.

\subsection{Weakly special and Mobius subvarieties}
A {\it weakly special subvariety\/} of
$Y(1)^n$ is an irreducible (over $\C$) component of a closed algebraic subset
of $Y(1)^n$ defined by an arbitrary finite collection of equations of the form
$\Phi_N(x_h, x_k)=0$ and $x_\ell =c$, a complex constant.

We will also refer to {\it weakly special subvarieties\/} of $\Hh^n$. These are
(the intersections with $\Hh^n$ of) varieties defined
by  an arbitrary collection of equations of the form $z_h=gz_k$, where
$k,h\in \{1,\ldots, n\}$ are distinct and $g\in {\rm GL}_2^+(\Q)$ or $z_\ell=c$,
a complex constant.

The image under $j$ of a weakly special subvariety of $\Hh^n$ is a weakly
special subvariety of $Y(1)^n$.

In $\Hh^n$ we will make use of a larger class of {\it Mobius subvarieties\/},
which are defined analogously to weakly special subvarieties except that the
matrices $g$ are allowed to be any element of ${\rm SL}_2(\R)$.

\subsection{The differential equation satisfied by $j(z)$.}
The  function $j(z)$ satisfies a nonlinear third order algebraic
differential equation, and none of lower order \cite{Ma}.
Specifically (see e.g. \cite{M}),
$$
F(j, j', j'', j''')=
Sj+{j^2-1968j+2654208\over 2j^2(j-1728)^2}(j')^2=0,
$$
where $Sf$ denotes the {\it Schwarzian derivative\/}
$
Sf={f'''\over f'}-{3\over 2}\Big({f''\over f'}\Big)^2
$
and $'$ indicates differentiation with respect to $z$.

\subsection{Proof of equivalence of Theorem \ref{main} and Theorem \ref{main2}}

We introduce a further
piece of notation. If $f_1,\ldots, f_n$ are complex functions of some variables
$w_1,\ldots, w_m$, we denote by $\dim(f_1,\ldots, f_n)$ the (complex)
dimension of the locus $\{f_1(\overline{w}), \ldots, f_n(\overline{w}))\}$
as $\overline{w}$ varies over a small open disk. If $f_i$ belong to a
differential field of meromorphic functions, with derivations $D_i$
corresponding to differentiation w.r.t. $w_i$, then
$\dim(f_1,\ldots, f_n)={\rm rank}(D_kf_i)$. Similarly, the transcendence
degree ${\rm tr. deg.\/}_{\C}\C(f_1,\ldots, f_n)$ is the dimension of the
Zariski closure of the same locus, which we denote
$\dim {\rm Zcl\/}(f_1,\ldots, f_n)$.

\begin{proof}

We first prove that  \ref{main} implies  \ref{main2}.
By the Seidenberg embedding theorem, we can consider $C=\C$,
the $z_i$ as functions of  $r=\textrm{rank}\left(D_k z_i\right)_{i,k}$
holomorphic variables $\tau_i$, and after possibly acting by an element of
$\SL_2(\C)$ on the $z_i$ --- which does not change the
transcendence degree --- we can assume $j_i = j(z_i)$. Let
$U\subset J_2\Hh^n\times J_2Y(1)^n$ denote, in local coordinates, the graph of
$$
\big(z_1,1,0,j(z_1),j'(z_1),j''(z_1), \ldots, z_n,1,0,j(z_n),j'(z_n),j''(z_n)\big).
$$
Note that here and below $j'$ denotes the derivative of $j:\Hh\rightarrow \C$
with respect to its argument, and $j'(z_i)$ denotes the composition
of $j'$ with $z_i$, not the derivative of $j(z_i)$ with respect to $z_i$.

Now applying Theorem \ref{main} we get the
conclusion of Theorem 1.3.

For the converse, parameterize $U$ locally as
$(\vec{z}, \vec{w}, \vec{r}, \vec{j}, \vec{j_1}, \vec{j_2})$
where we set $\vec{z}=(z_1,\ldots, z_n)$,
$\vec{w}=(w_1,\ldots, w_n)$, $\vec{r}=(r_1,\ldots, r_n)$,
$$\vec{j}=(j(z_1),\ldots, j(z_n)), \quad
\vec{j_1}=(j'(z_1)w_1,\ldots, j'(z_n)w_n),$$
and
$$\vec{j_2}=\big(\frac12\cdot j''(z_1)w_1^2 + j'(z_1)r_1,\ldots,
\frac12\cdot j''(z_n)w_n^2 + j'(z_n)r_n\big).$$
\noindent
We must prove that
$$ 
\dim{\rm Zcl\/}
(\vec{z}, \vec{w}, \vec{r}, \vec{j}, \vec{j_1}, \vec{j_2})\geq
3n + \dim(\vec{z}, \vec{w}, \vec{r}).$$
As all the $w_i\neq 0$ by assumption, this is equivalent to
$$ 
\dim{\rm Zcl\/}(\vec{z}, \vec{w}, \vec{r}, \vec{j}, \vec{j'}, \vec{j''})
\geq 3n + \dim(\vec{z}, \vec{w}, \vec{r})$$
where
$$
\vec{j'}=(j'(z_1),\ldots, j'(z_n)),\quad \vec{j''}=(j''(z_1),\ldots, j''(z_n)).
$$
By Theorem \ref{main2} we get that
$$ 
\dim{\rm Zcl\/}(\vec{z}, \vec{j}, \vec{j'}, \vec{j''})\geq 3n + \dim(\vec{z}).$$
We now obtain the conclusion of Theorem \ref{main}, since
$$ 
\dim{\rm Zcl\/}(A,B)-
\dim{\rm Zcl\/}(A)
\geq \dim(A,B)-\dim(A),$$
for any 2 finite sets of functions $A,B$.
\end{proof}

\subsection{Deduction of 1.1 from 1.2.\/}
We have shown that 1.3 implies 1.2.
Now 1.2 clearly implies the weakened conclusion (under the same hypotheses)
that
$$
\dim{\rm Zcl\/} (z_1, j_1, \ldots, z_n, j_n)\ge n+ \dim(z_1,\ldots, z_n).
$$
Given the hypotheses of 1.1, if we parameterize $U$ by $(z_i, j_i)$ we then
get the conclusion of 1.1 from the above.

\section{The fundamental domains that $U$ goes through}

Throughout we let $U_1$ denote the projection of $U$ to $\Hh^n$.
Without loss of generality we may assume that $U_1$ has non-empty intersection with the interior
of $\FF^n$.
For an irreducible complex analytic variety $S\subset\Hh^n$
we denote by $S^{\rm Mob}$ the smallest Mobius subvariety containing $S$.

\begin{definition*} For an complex analytic variety
$X\subset J_2\Hh^n\times J_2Y(1)^n$,
a positive integer $d$, and a Mobius subvariety $M\subset\Hh^n$,
we define $G_d(X,M)$ to be the set of points $x\in X$ around which $X$
is regular of dimension $d$, and such that the irreducible
component $X_0$ of the projection of $X$ to $\Hh^n$
containing $x$ satisfies $X_0^{\rm Mob}=M$.
\end{definition*}

We make the following definition:
$$Z=\{\gamma\in \SL_2(\R)^n\mid
\dim G_{\dim U_{1}}\left((\gamma^{-1}\cdot V)\cap (\Gamma\cap
\overline{\FF}^n),U_1^{\Mob}\right) = \dim U_{1}\}$$
so that $Z$ is a definable set.
The restriction on $U_{1}^{\rm Mob}$ ensures that the
condition that $U_{1}$ is not contained in a proper weakly
special subvariety is preserved under
possible modifications of $V, U$.
Note that, due to the invariance of $\Gamma$ under ${\rm SL}_2(\Z)^n$,
the integral points of $Z$ include those
$\gamma\in{\rm SL}_2(\Z)^n$ such that
$U_{1}$ is present in the fundamental domain
$\gamma\overline{\mathcal{F}}^n$, so that
$$\forall\gamma\in \SL_2(\Z)^n, \quad
\gamma\FF^n\cap U_1\neq\emptyset\rightarrow \gamma\in Z(\Z).$$

Now write $Z=\bigcup_{i=1}^{m} Z_i$ where each $Z_i$ is an irreducible
real-analytic variety (see \cite{DM}), and $Z_1$ contains the identity.
By replacing $V$ by
$\gamma^{-1} V$ by an appropriate $\gamma\in Z$ we assume that
$\dim Z_1\geq\dim Z_i$ for all $i$, and the identity $\mathbf{1}$ is a
smooth point of $Z_1$.

Now, define $S$ to be the stabilizer of $V$.
We show that by possibly modifying $V$ and $U$ a little
(but preserving the hypotheses of \ref{main} and the dimensions
of both $V$ and $U$) we may assume
certain regularity properties of $Z$.

\begin{lemma}\label{massage1}
By replacing $V$ by $\gamma_0^{-1}\cdot V$ for a generic
(measure theoretically) set of elements $\gamma_0\in Z_1$
we can ensure that
$S(\Z)\subset\gamma S\gamma^{-1}$ for all $\gamma\in Z_1$.
Moreover, we can ensure that, for all
$k\neq 1$, for each element $t\in Z_k(\Z)$ we have $t\in \gamma^{-1}\cdot Z_k$
for any $\gamma$ in an open neighbourhood of the identity in $Z_1$.
\end{lemma}

\begin{proof}

Consider an element $t\in SL_2(\Z)^n$,
and define the set
$$Z_t:=\{\gamma\in Z_1\mid t\in\gamma^{-1} S\gamma\}.$$
Then either $Z_t=Z_1$ or $Z_t$ has measure $0$ in $Z_1$.
Take $Z_1^0\subset Z_1$ to be in the complement of all $Z_t$ for
which $Z_t\neq Z_1$. Since there are only countable many integral points,
$Z_1^0$ has full measure in $Z_1$. Now pick $\gamma_0\subset Z_1^0$.

For the second part of the lemma, the same argument works once we notice
that replacing $V$ by $\gamma_0^{-1} V$ replaces $Z$ by $\gamma_0^{-1} Z$.
\end{proof}

From now on we assume $V$ is as in the conclusion of Lemma \ref{massage1}. Let $G$ be the Zariski closure of $S(\Z)$, and $N(G)$ the normalizer of $G$.

\begin{prop}\label{orbitdomain}

The set $Z(\Z)$ is --- and in particular the elements ${\gamma}\in\SL_2(\Z)^n$
such that $U_{1}$ goes through the fundamental domain
${\gamma}\cdot\overline{\FF}^n$
are --- contained in a finite union of right $N(G)(\Z)$ orbits.

\end{prop}

\begin{proof}

By the above lemma elements of $\gamma \in Z_1(\Z)$ satisfy
$S(\Z)\subset \gamma^{-1} S(\Z)\gamma$. It follows that $Z_1(\Z)\subset N(G)$.

Next, fix $k\neq 1$, take a point $t\in Z_k(\Z)$. By the above lemma,
$\gamma\cdot t\subset Z_k$ for $\gamma$ in a open neighbourhood
of the identity in $Z_1$. Since $\dim Z_1\geq\dim Z_k$, this implies that if we take
$E$ to be an open neighborhood of $Z_1$, then $Et$ is open in $Z_k$.
Now, since we saw that $S(\Z)\subset e^{-1}Se$ for all $e\in E$ it follows by analytic
continuation that for all $z\in Z_k$ we have that
$$S(\Z)\subset tz^{-1}Szt^{-1}.$$ It follows that for $z\in Z_k(\Z)$ we have
$z\in N(G)(\Z)t$. Since there are only finitely many $Z_k$, the proof is finished.
\end{proof}

\bigbreak

\section{Volume Estimates and Polynomial Presence}

\medbreak

%
%

\begin{lemma}
For any $\gamma\in \SL_2(\Z)^n$ the quantity
${\rm Vol\/}(\gamma\cdot U_1\cap \FF^n)$ is uniformly bounded above.
\end{lemma}

\begin{proof}
We will prove a stronger statement. Note that
$$\gamma U_1\cap \FF^n \subset \pi_{\Hh^n}
(\gamma\cdot V\cap \Gamma)\cap \FF^n.$$
and so
$$\Vol(\gamma\cdot U_1\cap \FF^n) \leq \sum_{\substack{ I\subset\{1,\dots,n\}\\
|I|=\dim U}} \deg(\pi_I)$$
where $\pi_I$ is the degree of the projection onto the $I$'th co-ordinates of
$$\gamma\cdot V\cap \Gamma\cap \overline{\FF}^n.$$
But now we can define this for any $\gamma\in Z$, and since this degree
is a definable function taking integer values it is uniformly bounded.
This completes the proof.
\end{proof}

Let $p\in U_1$ be a fixed point, and consider the geodesic ball $B(p,r)$
around $p$ of radius $r$
for the hyperbolic metric. By a theorem of Hwang-To \cite{HT} the volume of
$U_1\cap B(p,r)$ grows at least like $c_1c_2^r$, for constants
$c_1>0$ and $c_2>1$. In view of the previous
Lemma, $U_1\cap B(p,r)$ must intersect exponentially (in $r$) many
fundamental domains, all of which have height exponential (in $r$).
It follows by Proposition 5.2 of \cite{P1} that
$Z$ contains polynomially many integer points
(meaning that the number up to height $T$ is $\gg T^\delta$ for some positive
$\delta$ and implied constant). Therefore, by
the Counting Theorem \cite{PW},
we learn that $Z$ contains semi-algebraic curves.
Let $C_{\R}$ be such a curve, with $C$ the complex algebraic
curve containing $C_{\R}$. For each $c$ in an open neighborhood
containing $C_{\R}$, let $U_c$ denote a
component of $c\cdot V\cap \Gamma$ such that $\dim U_c=\dim U$
and $(\pi_{\Hh^n}U_c)^{\textrm{Mob}} = U_1^{\textrm{Mob}}$.
Let $c_0\in C_{\R}$.
We now have three cases:

\begin{itemize}

\item[(i)] $c\cdot V$ is not independent of $c\in C_{\R}$, and $c\cdot V$
does not contain
$U_{c_0}$ for all $c\in C_{\R}$. Then we may take $V' = C\cdot V$ to
get a new pair $V', U'$ where $U'$ is a component of $\cup_{c\in C} U_c$ increasing the dimension of each of $U$ and $V$ by one.

\item[(ii)] $c\cdot V$ is not independent of  $c\in C$ but $c\cdot V$ contains $U_{c_0}$ for all
$c\in C$. Then we replace $V, U$ by $V'=V\cap c\cdot V, U'=U_{c_0}$
for a generic $c\in C$,
which reduces the dimension of $V$ by one but leaves the dimension of $U$ unchanged.

\item[(iii)] $c\cdot V$ is independent of $c\in C$. Then $C_{\R}$ is contained in a right coset of $S$.

\end{itemize}

\medskip
\noindent
{\it Start of proof of Theorem 1.2.\/}
We argue by induction (upwards) on $\delta=\dim V-\dim U$, the case
$\delta=0$ being ``Ax-Lindemann'' with derivatives \cite{P2}.
For a given $\delta$ we argue by induction (downwards) on $\dim U$,
the case $\dim U=\dim\Gamma=3n$ being obvious.
If $C\subset Z$ is a semi-algebraic curve as above,
we either reduce $\delta$ if we are in case (ii),
or increase $\dim U$ maintaining the value of $\delta$
if we are in case (i), or find ourselves in case (iii).

Therefore, by induction, we may assume from now on that
we are always in case (iii), and that the conclusion of Lemma 3.1 holds.
Then all maximal algebraic sets in $Z$
having integer points are
right cosets of $S$.
By the Counting Theorem there are fewer than
polynomially many  of these cosets which contain integer points
(or we would enlarge $S$); therefore $S(\Z)$ must contain
polynomially many points, and so $G$ has positive dimension.
The next section deals with this case.

\section{Group Theory: Splitting of co-ordinates}
Throughout this entire section we assume we are in case (iii)
from section 4.
Consider the Lie Algebra $L(G)$. Let the semisimple part be $L(G)_{ss}$,
which must be isomorphic to $\sl_2^k$. Note that these $\sl_2$'s are
defined over $\Q$ as the $\Q$-points are Zariski dense in $G$.
Next set $\n$ to be the unipotent radical of $L(G)$.
Consider the projection maps $\pi_i:L(G)\rightarrow \sl_2$.

\begin{lemma}
For each $i$, $\pi_i(\n)$ and $\pi_i(L(G)_{ss})$ are not both non-zero.
In particular, $\n$ and $L(G)_{ss}$ commute.
\end{lemma}

\begin{proof}
For each $i$, $\pi_i(\n)$ is either empty, or  1-dimensional, whereas
$\pi_i(L(G)_{ss})$ is either empty or all of $\sl_2$. Since $\pi_i(\n)$ is
normal in $\pi_i(L(G)_{ss})$, the claim follows.
\end{proof}

%
%
%
%

Let $I_{ss}$ be the set of $i$ for which $\pi_i(L(G)_{ss})$ is non-empty, $I_{n}$
be the set of $i$ for which $\pi_i(n)$ is empty, and $I_e$ be the rest of the
co-ordinates. By conjugating by elements of $\SL_2(\Z)$ we may (and do)
assume that $n$ is contained entirely in the upper triangular matrices.
Now, we study $N(G)$. For the semi-simple part note that
$\pi_{I_{ss}}(L(Z(G)))=\mathbf{1}$
 which implies that
\begin{equation}\label{selfnorm}
\pi_{I_{ss}}(L(N(G)))=\pi_{I_{ss}}(L(G)).
\end{equation}

Thus, denoting by $G_{ss}$ the connected part of the semi-simple part
of $G$, we see that $G_{ss}$ is supported only on the $I_{ss}$ co-ordinates,
and $\pi_{I_{ss}}G_{ss}$ is of finite index in its own normalizer.

For the unipotent part, note that $\pi_{I_{n}}(L(N(G)))$ is contained
in the upper triangular Borel. However, the only integral points
in the Borel are unipotent, which means that
$\pi_{I_{n}}(N(G))(\Z)$
contains the upper triangular unipotents $M(\Z)\cong \Z^{|I_n|}$
as a subgroup of finite index. Set $G_n$ to be the unipotent radical of $G$,
which we know is only supported on the $I_n$ co-ordinates.
Take $G'_n\subset 1_{I_r\cup I_{ss}}\times M(\R)_n$ to be such that
$G'_n+G_n = 1_{I_r\cup I_{ss}}\times M(\R)_n$.

\begin{lemma}
$G_n(\Z)\backslash \pi_{I_n\cup I_r}(U_1)\subset G_n(\Z)\backslash
\Hh^{I_n\cup I_r}$ is complex analytic.
\end{lemma}

\begin{proof}
Since $U_1$ is $G(\Z)$ invariant, we know that
$$G_n(\Z)\backslash\pi_{I_n\cup I_r}(U_1) =
\textrm{im}\bigcup_{\gamma\in G_n\times G_{ss}(\Z)\backslash Z(\Z)}
\pi_{I_n\cup I_r}\left(U_1\cap \gamma\cdot(\FF^{I_{ss}}
\times \Hh^{I_n\cup I_r})\right).$$
By Proposition \ref{orbitdomain} and equation (\ref{selfnorm}) the RHS
is a finite union.
Therefore the image is closed, and so it is complex analytic
by Theorem 7.4 of \cite{PEST}  (and definable).
\end{proof}

\medskip
\noindent
{\it Continuation of proof of Theorem 1.2.\/} Set
$$S_n(R):=\{{z}\subset\Hh^{I_n}\mid {\rm Re}({z})\in [-R,R]^{I_n}\},$$
and let $B_r(R)\subset\Hh^{I_r}$ denote the hyperbolic ball of radius $R$
around an arbitrary fixed point, say $\vec{i}=(i,\ldots,i)$. Note that
in this ball the norms of the co-ordinates are bounded above by $e^{2R}$
while the $y$ co-ordinates are bounded below by $e^{-R}$,
so by \cite[Prop 5.2]{P1} the heights of all fundamental domains are
bounded from above by $e^{mR}$ for some fixed  $m$.

Consider $$A(R):=G(\Z)\backslash S_n(R)\times B_r(M\ln R))$$ for $M=|I_n|+1$.
Set $U_2= G_n(\Z)\backslash\pi_{I_n\cup I_r}(U_1)$. We now split into 2 cases,
depending on whether $U_2$ intersects the boundary of $A(R)$
for arbitrarily large $R$:

\textbf{Case A}: For arbitrarily large $R$,
$$U_2\cap\partial (G_n(\Z)\backslash S_n(R))\times B_r(M\ln R)\neq\emptyset.$$
In this case, since $U_2$ is connected and $G_n(\Z)\backslash\R^n$ is a
lattice partition in which one must pass at least $R$ boxes to get to the
boundary of $S_n(R$), there must be polynomially many points
$t_i\in Z\subset\SL_2(\Z)^n$ of height at most $O(R^{M+1})$ such that
the $t_i$ are in distinct $G(\Z)$-orbits, since $\pi_n(t_i)$ are in distinct
$G_n(\Z)$ orbits. However, this contradicts our present assumptions
which imply that the maximal semi-algebraic sets in $Z$
are orbits of $S$. So we must be in:

\textbf{Case B}: For sufficiently large $R$,
$$U_2\cap\partial (G_n(\Z)\backslash S_n(R))\times B_r(M\ln R)=\emptyset.$$
This means that $\pi_{I_r}(U_2\cap A(R))\subset B_r(M\ln R)$
is globally analytic. Suppose $\pi_{I_r}(U_2)$ is not a point.
Applying Hwang-To,
we see that we get at least $e^{MR}$ volume, which means
that the volume of $U_2$ was at least that big, so we hit at least that many
fundamental domains. At most $R^{|I_n|}$ of them could have come from the
first co-ordinate so we get polynomially many (in $e^R$) in the
second co-ordinate.
But this is a contradiction since the group $G$ is just a torus in that co-ordinate.
Thus $\pi_{I_r}(U_2)$ is a point. This implies by our earlier assumption that the
$I_r$ co-ordinates are empty. Hence, since $U_2$ also does not intersect the
boundary of $A(R)$, we must have that $U_2$, and hence
$U_{1}$, has only finitely many `pictures' in fundamental domains.
I.e. there are finitely many fundamental domains of the form
$\gamma_i\mathcal{F}^n, \gamma_i\in {\rm SL}_2(\Z)^n, i=1,\ldots, k$
such that for any  such fundamental domain $\gamma\mathcal{F}^n$
there is $i\in\{1,\ldots, k\}$ such that
$U_{1}\cap\gamma\mathcal{F}^n=
\gamma\gamma_i^{-1}(U_{1}\cap \gamma_i\mathcal{F}^n)$.
This implies that $\pi_{Y(1)^n}(U)$
is a complex analytic and definable subset of $Y(1)^n$, an affine space.
By Peterzil-Starchenko \cite{PS1}, this means that $\pi_{Y(1)^n}(U_1)$
is an algebraic subvariety.

\section{Monodromy and conclusion}

Now we have reduced ourselves to the case that $\pi_{Y(1)^n}U$ is algebraic.
The proof of Theorem 1.2 is therefore completed by the following result.

\begin{lemma}\label{YYY}
Theorem \ref{main} holds with the additional assumption that
$\pi_{Y(1)^n}U$ is algebraic.
\end{lemma}

\begin{proof}

We can clearly assume that $V=U^{\rm zar}$.
By a theorem of Andr\'e and Deligne \cite{An, De},
since $\pi_{Y(1)^n}U$ is not contained in a weakly special subvariety,
the monodromy of $\pi_{\Hh^n}U$ is Zariski dense in $\SL^n_2(\C)$.
This implies that the same is true for $\pi_{J_2Y(1)^n}U$, since the
image of the fundamental group induced by the projection map from
$\pi_{J_2Y(1)^n}U$ to $\pi_{Y(1)^n}U$ has finite index.

This means that $U$ is invariant under a Zariski dense subgroup of
$\SL_2^n(\C)$, and thus that $V=U^{\rm zar}$ is invariant under all of
$\SL_2^n(\C)$.

Finally, observe that $$J_2\PP^1\setminus J_2^0\PP^1$$
is acted on transitively by  $\SL_2\C$.
This implies that the fibers of the natural projection map $V\rightarrow J_2Y(1)^n$
are $J_2\C^n$, and so $3n$ dimensional, which implies the claim.
\end{proof}

\section{Application to Zilber-Pink}

A {\it special subvariety\/} of $Y(1)^n$ is a weakly
special subvariety (as defined in \S2.3) in which the constant $c$ in any
defining equation of the form $x_\ell=c$ is a {\it singular modulus\/},
i.e. the $j$-invariant of a CM elliptic curve.

\medskip

\begin{definition*}
Let $W\subset Y(1)^n$. An {\rm atypical subvariety\/} of $V$
is a component $A$ of some $W\cap T$, where $T$ is a special subvariety,
such that $\dim A >\dim W+\dim T-\dim X$.
\end{definition*}

\begin{conj}[Zilber-Pink conjecture for $Y(1)^n$]
Let $W\subset Y(1)^n$. Then $W$ contains
only finitely many maximal atypical subvarieties.
\end{conj}

It is shown in \cite{HP2}
that this conjecture would follow from two statements. The first is (a less
general form of) the Ax-Schanuel statement proved here.
The second is a conjectural lower bound for certain Galois orbits.
The purpose of this section is to show that Ax-Schanuel directly implies
a partial result on the conjecture, without any assumptions on Galois orbits.

\begin{definition*}

An atypical subvariety $A\subset W$ is called {\rm strongly atypical\/}
if no coordinate is constant on $A$.

\end{definition*}

\begin{thm}
Let $W\subset Y(1)^n$. Then $W$ contains only finitely many
maximal strongly atypical subvarieties.
\end{thm}

\begin{proof}
For a subvariety $A\subset Y(1)^n$ denote by $\langle A\rangle$ the
smallest special subvariety containing $A$, and denote by
$\delta(A)=\dim\langle A\rangle-\dim A$ the {\it defect\/} of $A$.
For $W\subset Y(1)^n$, we call $A\subset W$ an {\it optimal subvariety\/}
if $\delta(B)<\delta(A)$ for all strictly larger $B\supset A$ with $B\subset W$.
As shown in \cite{HP2}, the Zilber-Pink conjecture for $Y(1)^n$ is
equivalent to the assertion that, for $W\subset Y(1)^n$, $W$
contains only finitely many optimal subvarieties.

Now let $W\subset Y(1)^n$. We consider the uniformization
$\pi: \Hh^n\rightarrow Y(1)^n$ and the definable set
$Z=\pi^{-1}(W)\cap\mathcal{F}^n$. An atypical intersection $W\cap T$ gives
rise to an atypical intersection of $Z$ with a weakly special subvariety
of $\Hh^n$. We consider the definable family of Mobius subvarieties.
Among these we can define the subset of $M$
which intersect optimally (in the sense $\dim M-\dim (M\cap Z)$
can't be maintained for any strictly larger $M$) and
with $M\cap Z$ having no constant coordinate (which is
preserved under taking larger $M$). By Ax-Schanuel
in the form of 5.12 of \cite{HP2}, such $M$ are weakly special.
As they have no constant coordinates
they are in fact special, defined only by ${\rm GL}_2^+(\Q)$ conditions,
and the definable subset of these must be a finite set.
\end{proof}

We remark that the corresponding statement for the Andr\'e-Oort conjecture
(``A subvariety $W\subset Y(1)^n$ contains only finitely many
strongly special subvarieties'') follows in a similar way from the
Ax-Lindemann theorem for $Y(1)^n$ in \cite{P1},
though this was not highlighted there.

\bigskip
\noindent
{\bf Acknowledgements.\/} The authors thank Gareth Jones for a correction.
JP  thanks the EPSRC for supporting
his research under a grant entitled ``O-minimality and diophantine geometry'',
reference EP/J019232/1, and JT thanks NSERC for the support of his research
through a discovery grant.

\end{document}